\documentclass[12pt]{amsart}
\usepackage{amsmath}
\usepackage{hyperref}
\usepackage{amssymb}
\usepackage{latexsym}
\usepackage{amscd}
\usepackage{graphicx} 
\usepackage{amsthm}
\usepackage{mathrsfs}
\usepackage{xypic}
\usepackage{bm}

\newdimen\AAdi%
\newbox\AAbo%
\def\AAk#1#2{\s_etbox\AAbo=\hbox{#2}\AAdi=\wd\AAbo\kern#1\AAdi{}}%
\def\AAr#1#2#3{\s_etbox\AAbo=\hbox{#2}\AAdi=\ht\AAbo\raise#1\AAdi\hbox{#3}}%
\font\tenmsb=msbm10 at 12pt \font\sevenmsb=msbm7 at 8pt
\font\fivemsb=msbm5 at 6pt
\newfam\msbfam
\textfont\msbfam=\tenmsb \scriptfont\msbfam=\sevenmsb
\scriptscriptfont\msbfam=\fivemsb
\def\Bbb#1{{\tenmsb\fam\msbfam#1}}
\newtheorem{thm}{Theorem}[section]
\newtheorem{lem}{Lemma}[section]
\newtheorem{cor}{Corollary}[section]
\newtheorem{rem}{Remark}[section]
\newtheorem{pro}{Proposition}[section]

\newcommand{\ba}{\begin{array}}
\newcommand{\ea}{\end{array}}

\newcommand{\Section}[2]{\setcounter{equation}{0}
\allowdisplaybreaks
\section[#1]{#2}}

\def\n{\nabla}
\def\bn{\overline\nabla}
\def\ir#1{\mathbb R^{#1}}

\def\f#1#2{\frac{#1}{#2}}

\def\grs#1#2{\bold G_{#1,#2}}

\def\dt#1{\frac {d\,#1}{d\,t}}
\def\mc#1{\mathcal{#1}}

\def\pd#1#2{\frac {\partial #1}{\partial #2}}

\def\a{\alpha}
\def\be{\beta}

\def\p#1{\partial #1}

\def\de{\delta}
\def\De{\Delta}
\def\e{\eta}

\def\eps{\epsilon}
\def\G{\Gamma}
\def\g{\gamma}

\def\la{\lambda}

\def\Om{\Omega}
\def\th{\theta}

\def\w{\wedge}

\def\R{\Bbb{R}}

\def\bn{\bar{\nabla}}

\subjclass{58E20,53A10.}

\begin{document}
\title
[Some results on space-like self-shrinkers] {Some results on
space-like self-shrinkers}

\author
[Huaqiao Liu and Y. L. Xin]{Huaqiao Liu and Y. L. Xin}
\address{School of Mathematics and Information Sciences, Henan University, Kaifeng 475004, China.}
\email{liuhuaqiao@henu.edu.cn}
\address {Institute of Mathematics, Fudan University,
Shanghai 200433, China.} \email{ylxin@fudan.edu.cn}

\thanks{The  authors  are supported partially by NSFC}

\begin{abstract}
We study space-like self-shrinkers of dimension $n$ in
pseudo-Euclidean space $\ir{m+n}_m$with index $m$. We derive drift
Laplacian of the basic geometric quantities and obtain their
volume estimates in pseudo-distance function. Finally, we prove
rigidity results under minor growth conditions in terms of the mean curvature or
the image of Gauss maps .
\end{abstract}
\maketitle

\Section{Introduction}{Introduction}

Let $\ir{m+n}_m$ be an $(m+n)$-dimensional pseudo-Euclidean space
with the index $m$. The indefinite flat metric on $\ir{m+n}_m$ is
defined by $ds^2=\sum_{i=1}^n (dx^i)^2-\sum_{\a=n+1}^{m+n}
(dx^{\a})^2.$ In what follows we agree with the following range of
indices
$$A,\,B,\,C,\cdots=1,\cdots,m+n;\,
i,\,j,\,k\cdots=1,\cdots, n;$$
$$s, t=1,\,\cdots,m;\,\a,\,\be,\cdots=n+1,\cdots,m+n.$$

Let $F:M\to \ir{m+n}_m$ be a space-like $n$-dimensional submanifold
in $\ir{m+n}_m$ with the second fundamental form $B$ defined by
$B_{XY}\mathop{=}\limits^{def.}\left(\bn_X Y\right)^N$ for $X, Y\in
\G(TM)$. We denote $(\cdots)^T$ and $(\cdots)^N$ for the orthogonal
projections into the tangent bundle $TM$ and the normal bundle $NM$,
respectively. For $\nu\in\G(NM)$ we define the shape operator
$A^\nu: TM\to TM$  by $A^\nu(V)= - (\bn_V\nu)^T.$ Taking the trace
of $B$ gives the mean curvature vector $H$ of $M$ in $\ir{m+n}_m$
and $H\mathop{=}\limits^{def.} \text{trace}(B) =B_{e_i e_i},$ where
$\{e_i\}$ is a local orthonormal frame field of $M.$ Here and in the
sequel we use the summation convention. The mean curvature vector is
time-like, and a cross-section of the normal bundle.

We now consider  a one-parameter family $F_t=F(\cdot, t)$ of
immersions $F_t:M\to \ir{m+n}_m$ with the corresponding images
$M_t=F_t(M)$ such that
\begin{equation}\begin{split}
\dt{}F(x, t)&=H(x, t),\qquad x\in M\\
F(x, 0)&=F(x)
\end{split}\label{mcf}
\end{equation}
are satisfied, where $H(x, t)$ is the mean curvature vector of $M_t$
at $F(x, t).$  There are many interesting results on mean curvature
flow on space-like hypersurfaces in certain Lorentzian manifolds
\cite{E1,E2, E3,E-H}. For higher codimension we refer to the
previous work of the second author \cite{X}.

A special but important  class of solutions to (\ref{mcf}) are
self-similar shrinking solutions, whose profiles, space-like
self-shrinkers, satisfy a system of quasi-linear elliptic PDE of the
second order
\begin{equation}\label{ss}
H = -\frac{X^N}{2}.
\end{equation}

Besides the Lagrangian space-like self-shrinkers \cite{C-C-Y, H-W,
D-X1}, there is an interesting paper on curves in the Minkowski plane \cite{H}. The present paper is
devoted to general situation on space-like self-shrinker.

For a space-like $n-$submanifold $M$ in $\ir{m+n}_m$ we have the
Gauss map $\g:M\to \grs{n}{m}^m$. The target manifold is a
pseudo-Grassmann manifold, dual space of the Grassmann manifold
$\grs{n}{m}$. In the next section we will describe its geometric
properties, which will be used in the paper.

Choose a Lorentzian frame field $\{e_i, e_\a\}$ in $\ir{m+n}_m$ with
space-like $\{e_i\}\in TM$ and time-like $\{e_\a\}\in NM$ along the
space-like submanifold $F:M\to \ir{m+n}_m$. Define coordinate
functions
$$x^i=\left<F, e_i\right>,\;y^\a=-\left<F, e_\a\right>.$$
We then have
$$|F|^2=X^2-Y^2,$$
where $X=\sqrt{\sum_{i=1}^n (x^i)^2},\quad
Y=\sqrt{\sum_{\a=n+1}^{m+n} (y^\a)^2}.$ We call  $|F|^2$ the
pseudo-distance function from the origin $0\in M$.

We always put the origin on $M$ in the paper. We see that $|F|^2$ is
invariant under the Lorentzian action up to the choice of the origin
in $\ir{m+n}_m$. Set $z=|F|^2$. It has been proved that $z$ is
proper provided $M$ is closed with the Euclidean topology (see
\cite{C-Y} for $m=1$ and \cite{J-X} for any codimension $m$).

Following Colding and Minicozzi \cite{C-M} we can also introduce the
drift Laplacian,
\begin{equation}\aligned
\mc{L}=\De-\f{1}{2}\langle
F,\n(\cdot)\rangle=e^{\f{z}{4}}div(e^{-\f{z}{4}}\n(\cdot)).
\endaligned\label{dL}\end{equation}
It can be showed that $\mc{L}$ is self-adjoint with respect to the
weighted volume element $e^{-\f{z}{4}}d\mu,$ where $d\mu$ is the volume
element of $M$ with respect to the induced metric from the ambient
space $\ir{m+n}_m$. In the present paper we carry out integrations
with respect to this measure. We denote $\rho=e^{-\f{z}{4}}$ and the
volume form $d\mu$ might be omitted in the integrations for
notational simplicity.

For a space-like submanifold in $\ir{m+n}_m$ there are several
geometric quantities. The squared norm of the second fundamental
form $|B|^2$, the squared norm of the mean curvature $|H|^2$ and the
$w-$function, which is related to the image of the Gauss map. In \S
3 we will calculate drift Laplacian $\mc{L}$ of those quantities,
see Proposition \ref{dLBHw}.

Corresponding to the weighted measure and drift Laplacian there is so-called
the Baker-Emery Ricci tensor. It is noted that in \cite{C-Q} $\text{Ric}_f\ge \f{z}{4}$ with $f=\f{z}{4}$.
Using the comparison technique the weighted volume of the geodesic ball can be estimated from above
in terms of the distance function \cite{W-W}.

For a space-like $n-$submanifold $M$ in $\ir{m+n}_m$ there are 3 kind global conditions: Closed one with Euclidean topology;
entire graph; complete with induced Riemannian metric. A complete space-like one has to be entire graph, but
the converse claim is not always the case. Closed one with Euclidean topologg is complete under the parallel mean curvature
assumption (see \cite{C-Y} for codimension one and \cite{J-X} for higher codimension).

In our case of closed one with Euclidean topology, the pseudo-distance function $z$ is always proper. It is
natural to consider the volume growth in $z$. For the proper
self-shrinkers in Euclidean space Ding-Xin \cite{D-X} gave the
volume estimates. It has been generalized in \cite{C-Z} for more
general situation. But, the present case does not satisfy the
conditions in Theorem 1.1 in \cite{C-Z}. However, the idea in
\cite{D-X} is still applicable for space-like self-shrinkers. In \S
4 we will give volume estimates for space-like self-shrinkers, in a
similar manner as in \cite{D-X}, see Theorem \ref{vg}.

Finally, using integral method we can obtain  rigidity results as
follows.

\begin{thm}
Let $M$ be a space-like self-shrinker of dimension $n$ in
$R^{n+m}_m,$ which is closed with respect to the Euclidean topology.
If there is a constant $\a<\f{1}{8}$, such that $|H|^2\le e^{\a z}$, then $M$
is an affine $n-$plane.
\end{thm}

\begin{thm}
Let $M$ be a complete space-like self-shrinker of dimension $n$ in
$R^{n+m}_m$. If there is a constant $\a<\f{1}{2}$, such that $\ln w\le e^{\a d^2(p, x)}$ for certain
$p\in M$, where $d(p,\cdot)$ is the distance function from $p$, then $M$ is affine $n-$plane.
\end{thm}

\begin{rem}
In the special situation, for the Lagrangian space-like
self-shrinkers, the rigidity results hold without the growth
condition (see \cite{D-X1}). Let $\R^{2n}_n$ be Euclidean space with null coordinates
$(x, y) = (x_1, \cdots, x_n; \, y_1, \cdots, y_n)$, which means that
the indefinite metric is defined by $ds^2=\sum_i dx_i dy_i.$ If
$M=\{(x,Du(x))\big|\ x\in\R^n\}$ is a space-like submanifold in
$\R^{2n}_n$, then $u$ is convex and the induced metric on $M$ is given by $ds^2=\sum_{i,j}u_{ij}dx_idx_j$.
$M$ is a space-like Lagrangian submanifold in $\ir{2n}_n$. It is worthy to point out that
if $M$ is entire gradient graph the potential function $u$ is proper, as the following consideration.
On $\ir{n}$ set $\rho=|x|=\sqrt{\sum x_i^2}$. At any direction $\th\in S^{n-1}$
$$u_i=u_\rho\pd{\rho}{x_i}=\f{x_i}{\rho}u_\rho$$
and the pseudo-distance
$$z=x_iu_i=\rho u_\rho,$$
which is positive when the origin is on $M$, since it is space-like. It implies that $u$ is increasing in $\rho$.
Moreover,
$$z_\rho=u_\rho+\rho u_{\rho\rho}>0,$$
which means that $z$ is also increasing in $\rho$. Hence,
$$u(\rho)-u(\eps)=\int_\eps^\rho u_\rho d\rho=\int_\eps^\rho\f{z}{\rho}d\rho\ge z(\eps)\int_\eps^\rho\f{1}{\rho}d\rho
\ge z(\eps)\int_\eps^\rho\f{1}{\rho}d\rho\to\infty$$
as $\rho\to\infty$.
\end{rem}

\begin{rem}
Rigidity problem for space-like extremal submanifolds was raised by
E. Calabi \cite{C}, and solved by Cheng-Yau \cite{C-Y} for
codimension $1$. Later, Jost-Xin generalized the results to higher
codimension \cite{J-X}. The rigidity problem for space-like
submanifolds with parallel mean curvature was studied in
\cite{X1}\cite{X-Y} and \cite{J-X} (see also in Chap. 8 of
\cite{X2}). \end{rem}

\bigskip

\Section{Geometry of $G_{n,m}^m$}{Geometry of $G_{n,m}^m$}

\medskip

In $\ir{n+m}_m$ all space-like $n-$subspaces form the
pseudo-Grassmannian $G^{m}_{n,m}.$ It is a specific Cartan-Hadamard
manifold which is the noncompact dual space of the Grassmann
manifold $G_{n,m}.$

Let $P$ and $A\in G_{n,m}^m$ be two space-like $n-$plane in
$R_m^{m+n}.$ The angles between $P$ and $A$ are defined by the
critical values of angel $\theta$ between a nonzero vector $x$ in P
and its orthogonal projection $x^*$ in $A$ as $x$ runs through $P$.
\par Assume that $e_1,\cdots, e_n$ are orthonormal vectors which
span the space-like $P$ and $a_1,\cdots,a_n$ for space-like $A.$ For
a nonzero vector in $P$
$$x=\sum_i x_i e_i,$$its orthonormal projections in $A$ is
$$x^*=\sum_ix_i^* a_i.$$ Thus, for any $y\in A,$ we have $$\langle
x-x^*,y\rangle=0.$$ Set $$W_{i\,j}=\langle e_i,a_j\rangle,$$ We then
have $$x_j^*=\sum_iW_{i\,j}x_i.$$ Since $x$ is a vector in a
space-like $n-$plane and its projection $x^*$ in $A$ is also a
space-like vector. We then have a Minkowski plane $R_1^2$ spanned by
$x$ and $x^*.$ Then angle $\theta$ between $x$ and $x^*$ is defined
by $$\cosh \theta=\f{\langle x,x^*\rangle}{|x||x^*|}.$$

Let$$W=(W_{i\,j})=\left(\begin{array}{lll}\langle
e_1,a_1\rangle&\cdots&\langle e_n,a_1\rangle\\ \quad\vdots&\
\vdots&\quad\vdots\\\langle e_n,a_1\rangle&\cdots&\langle
e_n,a_n\rangle\end{array}\right)$$ Now define the $w-$function as
$$w=\langle e_1\w\cdots\w e_n,a_1\w\cdots\w a_n\rangle=\det W.$$
$W^TW$ is symmetric, its eigenvalues are $\mu_1^2,\cdots,\mu_n^2,$
then there exist $e_1,\cdots,e_n$ in $P$, such that
$$W^TW=\left(\begin{array}{lll}\mu_1^2&&0\\\quad&\ddots&\quad\\0&&\mu_n^2\end{array}\right),$$
in which $\mu_i\ge1$ and $ \mu_i=\cosh\theta_i.$ Then
\begin{equation}\label{wla}
w=\prod_i\cosh\theta_i=\prod_i\f{1}{\sqrt{1-\la_i^2}},\;
\la_i=\tanh\theta_i.
\end{equation}
The distance between $P$ and $A$ in the canonical Riemannian metric
on $\grs{n}{m}^m$ is (see \cite{W} for example)
$$d(P, A)=\sqrt{\sum_i\th_i^2}.$$
For the fixed $A\in G_{n,m}^m,$ which is spanned by $\{a_i\}$,
choose time-like $\{a_{n+s}\}$ such that $\{a_i, a_{n+s}\}$ form an
orthonormal Lorentzian bases of $R^{n+m}_m.$

Set
\begin{equation*} \aligned
e_i&=\cosh\theta_i a_i+\sinh\theta_ia_{n+i}\\
e_{n+i}&=\sinh\theta_ia_i+\cosh\theta_ia_{n+i}\,(\;\text{and}\;
e_{n+\a}=a_{n+\a}\; \text{if}\; m>n) .\endaligned\end{equation*}
Then $e_i\in T_pM, e_{n+i}\in N_pM.$ In this
case\begin{eqnarray}\label{wial} w_{i\,\a}&=&\langle e_1\w\cdots\w
e_{i-1}\w e_\a\w e_{i+1}\cdots\w e_n,a_1\w\cdots\w
a_n\rangle\\&=&\cosh\theta_1\cosh\theta_{i-1}\sinh\theta_i\cosh\theta_{i+1}\cosh\theta_n=\la_i
w\delta_{n+i\ \a},\end{eqnarray} which is obtained by replacing
$e_i$ by $e_\a$ in $w$. We also have $w_{i\a j\beta}$ by replacing
$e_j$ by $e_\beta$   in $w_{i\,\a}.$ We obtain
\begin{eqnarray}\label{wialjbe}
w_{i\a
j\be}=\left\{\begin{array}{lll}\la_i\la_jw&&\a=n+i,\beta=n+j\\-\la_i\la_jw&&\a=n+j,\beta=n+i\\0&&otherwise.
\end{array}\right.\end{eqnarray}

\bigskip

\Section{Drift Laplacian of some geometric quantities}{Drift
Laplacian of some geometric quantities}

\medskip

The second fundamental form $B$ can be viewed as a cross-section of
the vector bundle Hom($\odot^2TM, NM$) over $M.$ A connection on
Hom($\odot^2TM, NM$) can be induced from those of $TM$ and $NM$
naturally. There is a natural fiber metric on Hom($\odot^2TM,
NM$) induced from the ambient space and it becomes a Riemannian
vector bundle. There is the trace-Laplace operator $\n^2$ acting on
any Riemannian vector bundle.

In \cite{X} we already calculate $\n^2 B$ for general space-like
$n-$submanifolds in $\ir{m+n}_m$.

Set $$B_{i\,j}=B_{e_i\,e_j}=h_{i\,j}^\a e_\a,\;
S_{\a\,\beta}=h_{i\,j}^\a h_{i\,j}^\beta.$$ From Proposition 2.1 in
\cite{X} we have
\begin{equation}\aligned\label{tLB}
\langle
\nabla^2B,B\rangle=\langle\nabla_i\nabla_jH,B_{i\,j}\rangle+\langle
B_{i\,k},H\rangle\langle
B_{i\,l},B_{k\,l}\rangle-|R^\perp|^2-\sum_{\a,\beta}S_{\a\,\beta}^2,\endaligned\end{equation}
where $R^\perp$ denotes the curvature of the normal bundle and
$$|R^\perp|^2=-\langle
R_{e_i\,e_j}\nu_\a,R_{e_i\,e_j}\nu_\a\rangle.$$

Then from the self-shrinker equation (\ref{ss}) we  obtain
\begin{eqnarray*}
\n_i F^N&=&[\bar{\n}_i(F-\langle F,e_j\rangle
e_j)]^N\\&=&[e_i-\bar{\n}_i\langle F,e_j\rangle e_j-\langle
F,e_j\rangle\bar{\n}_{e_i}e_j]^N\\&=&-\langle F,e_j\rangle B_{i\,j},
\end{eqnarray*}
and\begin{eqnarray*} \n_i\n_j F^N&=&-\n_i[\langle F,e_k\rangle
B_{k\,j}]\\&=&-\de_i^kB_{k\,j}-\langle F^N,B_{k\,i}\rangle
B_{k\,j}-\langle F,e_k\rangle\n_i B_{k\,j}\\&=&-B_{i\,j}-\langle
F^N,B_{k\,i}\rangle B_{k\,j}-\langle F,e_k\rangle\n_k
B_{i\,j}\\&=&-B_{i\,j}+\langle 2H,B_{k\,i}\rangle B_{k\,j}-\langle
F,e_k\rangle\n_k B_{i\,j}.\end{eqnarray*} Set $P_{i\,j}=\langle
B_{i\,j},H\rangle,$ then
\begin{equation}\aligned\label{hij}
\n_i\n_j H=\f{1}{2}B_{i\,j}-P_{k\,i}B_{k\,j}+\f{1}{2}\langle
F,e_k\rangle\n_k B_{i\,j}.
\endaligned
\end{equation}
Substituting (\ref{hij}) into (\ref{tLB})we obtain
\begin{eqnarray*}
\langle\nabla^2B,B\rangle&=&\langle\f{1}{2}B_{i\,j},B_{i\,j}\rangle-\langle
H,B_{k\,i}\rangle\langle B_{k\,j},B_{i\,j}\rangle+\f{1}{2}\langle
F,e_k\rangle\langle \nabla_k B_{i\,j},B_{i\,j}\rangle\\&&+\langle
B_{i\,k},H\rangle\langle B_{i\,l},
B_{k\,l}\rangle-|R^\perp|^2-\sum_{\a,\beta}S^2_{\a\,\beta}.\end{eqnarray*}
This also means that
\begin{equation}\aligned
\langle\nabla^2B,B\rangle=\f{1}{2}\langle B,B\rangle+\f{1}{4}\langle
F^T,\nabla\langle
B,B\rangle\rangle-|R^\perp|^2-\sum_{\a,\beta}S^2_{\a\,\beta}.\endaligned\label{3.7}
\end{equation}
Note that $\Delta\langle
B,B\rangle=2\langle\nabla^2B,B\rangle+2\langle\nabla B,\nabla
B\rangle,$ so\begin{equation}\aligned \Delta\langle
B,B\rangle&=\langle B,B\rangle+\f{1}{2}\langle F^T,\nabla\langle
B,B\rangle\rangle-2|R^\perp|^2-2\sum_{\a,\beta}S^2_{\a\,\beta}\\&+2\langle\nabla
B,\nabla B\rangle.\endaligned\label{3.8}
\end{equation}
We denote $$|B|^2=-\langle B,B\rangle=\sum_{i,j,\a}h_{\a ij}^2,\;
|\n B|^2=-\langle\n B,\n B\rangle.$$
$$|H|^2=-\langle H,H\rangle,\; |\n H|^2=-\langle \n H,\n H\rangle$$ then
\begin{equation}\aligned
\Delta|B|^2=|B|^2+\f{1}{2}\langle
F^T,\nabla|B|^2\rangle+2|R^\perp|^2+2\sum_{\a,\beta}S^2_{\a\,\beta}+2|\nabla
B|^2\endaligned\label{DeB}\end{equation}

From (\ref{hij}) we  also obtain
\begin{equation*}\aligned
\n^2 H=\f{1}{2}H-P_{k\,i}B_{k\,j}+\f{1}{2}\langle F,e_k\rangle\n_k
H.
\endaligned
\end{equation*}
Since\begin{eqnarray*} \De |H|^2=-\De\langle
H,H\rangle=-2\langle\n^2 H,H\rangle-2\langle\n H,\n
H\rangle,\end{eqnarray*} we obtain
\begin{equation}\aligned\label{DeH}
\De |H|^2&=-2\langle \f{1}{2}H-P_{k\,i}B_{k\,i}+\f{1}{2}\langle
F,e_k\rangle\n_k H,H\rangle-2\langle\n H,\n
H\rangle\\&=|H|^2+2|P|^2+\f{1}{2}\langle F^T,\n |H|^2\rangle+2|\n
H|^2,\endaligned
\end{equation}
where $|P|^2=\sum_{i, j} P_{ij}^2$.

In the pseudo-Grassmann manifold $\grs{n}{m}^m$ there are
$w-$functions with respect to a fixed point $A\in \grs{n}{m}^m$, as
shown in \S 2. For the space-like $n-$submanifold $M$ in
$\ir{m+n}_m$ we define the Gauss map $\g:M\to  \grs{n}{m}^m$, which
is obtained by parallel translation of $T_pM$ for any $p\in M$ to
the origin in $\ir{m+n}_m$. Then, we have  functions $w\circ\g$ on
$M$, which is still denoted by $w$ for notational simplicity.

For any point $p\in M$ around $p$ there is a local tangent frame
field $\{e_i\}$, and which is normal at $p$. We also have a local
orthonormal normal frame field $\{e_\a\}$, and which is normal at
$p$. Define a $w-$function by
$$w=\left<e_1\w\cdots\w e_n, a_1\w\cdots\w a_n\right>,$$
where $\{a_i\}$ is a fixed orthonormal vectors which span a fixed
space-like $n-$plane $A$. Denote $$e_{i\,\a}=e_1\w\cdots\w
e_\a\w\cdots \w e_n,$$ which is got by substituting $e_\a$ for $e_i$
in $e_1\w\cdots\w e_n$ and  $e_{i\a j\beta}$ is obtained by
substituting $e_\beta$ for $e_j$ in $e_{i\,\a}.$ Then
\begin{equation}\aligned
\nabla_{e_j}w&=\sum_{i=1}^n\langle
e_1\w\cdots\bar{\nabla}_{e_j}e_i\w\cdots\w e_n,a_1\w\cdots\w
a_n\rangle\\&=\sum_{i=1}^n\langle e_1\w\cdots\w B_{i\,j}\cdots
e_n,a_1\w\cdots\w a_n\rangle\\&=\sum_{i=1}^n h_{i\,j}^\a \langle
e_1\cdots\w e_\a\w\cdots\w e_n,a_1\w\cdots\w
a_n\rangle\\&=\sum_{i=1}^n h_{i\,j}^\a \langle
e_{i\,\a},a_1\w\cdots\w a_n\rangle.\endaligned\label{nw}
\end{equation}

Furthermore,
\begin{eqnarray}
\n_{e_i}\n_{e_j}w&=&\langle\bar{\n}_{e_i}\bar{\n}_{e_j}(e_1\w\cdots\w
e_n),a_1\w\cdots\w a_n\rangle\nonumber\\&=&\sum_{k\neq l}\langle
e_1\w\cdots\w \bar{\n}_{e_j}e_k\w\cdots\w\bar{\n}_{e_i}e_l\w\cdots\w
e_n,a_1\w\cdots\w a_n\rangle\nonumber\\&&+\sum_k\langle
e_1\w\cdots\bar{\n}_{e_i}\bar{\n}_{e_j}e_k\w\cdots\w
e_n,a_1\w\cdots\w a_n\rangle\nonumber\\&=&\sum_{k\neq l}\langle
e_1\w\cdots\w B_{j\,k}\w\cdots\w B_{il}\w\cdots e_n,a_1\w\cdots\w
a_n\rangle\label{3.13}\\&&+\sum_k\langle
e_1\w\cdots\w(\bar{\n}_i\bar{\n}_je_k)^T\w\cdots\w e_n,a_1\w\cdots\w
a_n\rangle\label{3.14}\\&&+\sum_k\langle
e_1\w\cdots\w(\bar{\n}_i\bar{\n}_je_k)^N\w\cdots\w e_n,a_1\w\cdots\w
a_n\rangle\label{3.15}\end{eqnarray}

Note that
\begin{eqnarray*}
(\ref{3.13})&=&\sum_{k\neq l}h_{j\,k}^\a h_{i\,l}^\beta\langle e_{\a
k\beta l},a_1\w\cdots\w a_n\rangle\\
(\ref{3.14})&=&\langle\bar{\n}_i\bar{\n}_j e_k,e_k\rangle w=-\langle
\bar{\n}_je_k,\bar{\n}_ie_k\rangle w=-\langle
B_{j\,k},B_{i\,k}\rangle w=h_{j\,k}^\a h_{i\,k}^\a
w\\
(\ref{3.15})&=&-\langle(\bar{\n}_i\bar{\n}_je_k)^N,e_\a\rangle\langle
e_{\a\,k},a_1\w\cdots\w
a_n\rangle\\
&=&-\langle(\bar{\n}_i(B_{j\,k}+\n_{e_j}e_k))^N,e_\a\rangle\langle
e_{\a\,k},a_1\w\cdots\w
a_n\rangle\\
&=&-\langle\n_iB_{j\,k},e_\a\rangle\langle e_{\a\,k},a_1\w\cdots\w
a_n\rangle=-\langle\n_kB_{i\,j},e_\a\rangle\langle
e_{\a\,k},a_1\w\cdots\w a_n\rangle,\end{eqnarray*} where we use the
Codazzi equation in the last step. Thus, we obtain
$$\De w=\sum_{i,k\neq l}h_{i\,k}^\a h_{i\,l}^\beta\langle e_{k\beta l}^\a ,a_1\w\cdots\w a_n\rangle+|B|^2w
-\langle\n_kH,e_\a\rangle\langle e_{\a k},a_1\w\cdots\w
a_n\rangle,$$ Since $$\n_i F^N=-\langle F,e_j\rangle B_{i\,j},$$
from (\ref{ss}), we obtain \begin{equation} \aligned
\n_iH&=\f{1}{2}\langle F,e_j\rangle
B_{i\,j}\\\langle\n_iH,e_\a\rangle&=-\f{1}{2}\langle F,e_j\rangle
h_{i\,j}^\a,\endaligned\label{nh}
\end{equation}
so,
\begin{equation}\aligned
\De w&=|B|^2w+\sum_{i,k\neq l}h_{i\,k}^\a h_{i\,l}^\beta\langle
e_{\a k\beta l},a_1\w\cdots\w a_n\rangle+\f{1}{2}\langle
F,e_i\rangle h_{k\,i}^\a \langle e_{\a k},a_1\w\cdots\w
a_n\rangle\\&=|B|^2w+\sum_{i,k\neq l}h_{i\,k}^\a
h_{i\,l}^\beta\langle e_{\a k\beta l},a_1\w\cdots\w
a_n\rangle+\f{1}{2}\langle F,\nabla w\rangle,\endaligned\label{Dew}
\end{equation}
where   (\ref{nw}) has been used in the last equality.

\begin{pro}\label{dLBHw}
For a space-like self-shrinker $M$ of dimension $n$ in $\ir{m+n}_m$
we have
\begin{equation}
\mc{L}|B|^2=|B|^2+2|R^\perp|^2+2\sum_{\a,\beta}S^2_{\a\,\beta}+2|\nabla
B|^2, \label{dLB}\end{equation}
\begin{equation}
\mc{L} |H|^2=|H|^2+2|P|^2+2|\n H|^2, \label{dLH}\end{equation}
\begin{equation}\label{dLlogw}
\mc{L}(\ln w)\ge\f{|B|^2}{w^2}.
\end{equation}
\end{pro}

\begin{proof}
From (\ref{dL},\ref{DeB},\ref{DeH}), we can obtain (\ref{dLB}) and
(\ref{dLH}) easily.

From (\ref{dL},\ref{Dew}) we have
\begin{equation}\aligned
\mc{L} w&=|B|^2w+\sum_{i,k\neq l}h_{i\,k}^\a h_{i\,l}^\beta\langle e_{\a k\beta l},a_1\w\cdots\w a_n\rangle
=|B|^2w+\sum_{i,k\neq l}h_{i\,k}^\a h_{i\,l}^\beta w_{\a k\beta l}\\&=
|B|^2w+\sum_{i,k\neq l}\la_k\la_l(h_{i\,k}^{n+k}h_{i\,l}^{n+l}-h_{i\,k}^{n+l}h_{i\,l}^{n+k})w,
\endaligned\label{dLw}\end{equation}

Furthermore, since
$$\mc{L}(\ln w)=\f{1}{w}\mc{L} w-\f{|\nabla w|^2}{w^2},$$
we obtain
$$
\mc{L}(\ln w)=|B|^2+\sum_{i,k\neq
l}\la_k\la_l(h_{i\,k}^{n+k}h_{i\,l}^{n+l}-h_{i\,k}^{n+l}h_{i\,l}^{n+k})-\f{|\nabla
w|^2}{w^2}.$$ From (\ref{nw}),we  obtain
$$\aligned
|\nabla
w|^2&=\sum_{j=1}^n|\nabla_{e_j}w|^2=\sum_{j=1}^n(\sum_{i=1}^n\sum_\a
h_{ij}^\a
w_{i\,\a})^2\\
&=\sum_{j=1}^n(\sum_{i=1}^nh_{ij}^{n+i}\la_iw)^2=\sum_{i,j,k=1}^n
\la_i\la_kw^2 h_{i\,j}^{n+i}h_{k\,j}^{n+k}.\endaligned$$ in the case
of $m\ge n$ we rewrite (otherwise, we treat the situation similarly)
$$|B|^2=\sum_{j,k,\a>n}(h^{n+\a}_{jk})^2+\sum_{i,j}(h_{ij}^{n+i})^2+\sum_j\sum_{k<i}(h_{ij}^{n+k})^2
+\sum_j\sum_{i<k}(h_{ij}^{n+k})^2.$$
So, we obtain
\begin{equation}\aligned
\mc{L}(\ln w)&=|B|^2+\sum_{i,j,k\neq i}\la_i\la_k(h_{i\,j}^{n+i}h_{j\,k}^{n+k}-h_{i\,j}^{n+k}h_{j\,k}^{n+i})-
\sum_{i,j,k=1}^n\la_i\la_kh_{ij}^{n+i}h_{j\,k}^{n+k}\\
&=|B|^2+\sum_{i,j,k\neq i}\la_i\la_kh_{i\,j}^{n+i}h_{j\,k}^{n+k}-\sum_{i,j,k\neq i}\la_i\la_kh_{i\,j}^{n+k}h_{j\,k}^{n+i}-\sum_{i,j,k=1}^n\la_i\la_kh_{ij}^{n+i}h_{j\,k}^{n+k}\\
&=\sum_{j,k,\a>n}(h^{n+\a}_{jk})^2+\sum_{i,j}(h_{i\,j}^{n+i})^2+\sum_{j}\sum_{k<i}(h_{ij}^{n+k})^2+\sum_j\sum_{i<k}(h_{ij}^{n+k})^2\\
&\hskip2in-\sum_{i,j}\la_i^2(h_{ij}^{n+i})^2-\sum_{ij,k\neq i}\la_i\la_kh_{ij}^{n+k}h_{jk}^{n+i}\\
&=\sum_{j,k,\a>n}(h^{n+\a}_{jk})^2+\sum_{i,j}(1-\la_i^2)(h_{ij}^{n+i})^2+\sum_j\sum_{k<i}(h_{ij}^{n+k})^2+\sum_j\sum_{i<k}(h_{ij}^{n+k})^2\\
&\hskip2.5in-2\sum_j\sum_{k<i}\la_k\la_ih_{jk}^{n+i}h_{ij}^{n+k}\\
&\ge\sum_{j,k,\a>n}(h^{n+\a}_{jk})^2+\sum_{i,j}(1-\la_i^2)(h_{ij}^{n+i})^2+\sum_j\sum_{k<i}(1-\la_i^2)(h_{ij}^{n+k})^2\\
&\hskip3.2in+\sum_j\sum_{i<k}(1-\la_i^2)(h_{ij}^{n+k})^2\\
&=\sum_{j,k,\a>n}(h^{n+\a}_{jk})^2+\sum_{i,j,k}(1-\la_i^2)(h_{ij}^{n+k})^2\\
&\ge\sum_{j,k,\a>n}(h^{n+\a}_{jk})^2+\prod_i(1-\la_i^2)\sum_{i,j,k}(h_{ij}^{n+k})^2.\endaligned\end{equation}
Noting (\ref{wla}) the inequality (\ref{dLlogw}) has been proved.

\begin{rem}
For a space-like graph $M=(x, f(x))$ with $f:\ir{n}\to\ir{m}$ its induced metric is $ds^2=(\de_{ij}-f^\a_if^\a_j)dx^idx^j.$ Set
$g=\det(\de_{ij}-f^\a_if^\a_j)$ then $w=\f{1}{\sqrt{g}}$.
\end{rem}
\begin{rem}
(\ref{dLlogw}) is a generalization of a formula (5.8) for space-like
graphical self-shrinkers in \cite{D-W} to more general situation.
\end{rem}

\end{proof}

\bigskip
\Section{Volume growth}{Volume growth}

\medskip

To draw our results we intend to integrate those differential
inequalities obtained in the last section. We need to know the
volume growth in the pseudo-distance function $z$ on the space-like
submanifolds.  In \cite{J-X} the following property has been
proved.

\begin{pro}(Proposition 3.1 in \cite{J-X})
Let $M$ be a space-like $n-$submanifold in $\ir{m+n}_m$. If $M$ is
closed with respect to the Euclidean topology, then when $0\in M$,
$z=\left<F, F\right>$ is a proper function on $M$.
\end{pro}
we also need a lemma from \cite{D-X}:
\begin{lem}\label{DXL}
If $f(r)$ is a monotonic increasing nonnegative function on
$[0,+\infty)$ satisfying $f(r)\le C_1r^nf(\f r2)$ on $[C_2,+\infty)$
for some positive constant $n,C_1,C_2$, here $C_2>1$,
 then $f(r)\le C_3e^{2n(\log r)^2}$ on $[C_2,+\infty)$ for some positive constant $C_3$ depending only on
  $n,C_1,C_2,f(C_2)$.
\end{lem}

 Using the similar method as in \cite{D-X} we obtain the following volume growth
 estimates.
\begin{thm}\label{vg}
Let $z=\langle F, F\rangle$ be the pseudo-distance of
$R_{m}^{n+m},$ where $F\in R_m^{n+m}$ is the position vector with respect to the origin $0\in M$. Let
$M$ be an $n-$dimensional space-like self-shrinker of $R^{n+m}_m$.
Assume that $M$ is closed with respect to the Euclidean topology,  then for any $\a>0$,
$\int_Me^{-\a z}d\mu$ is finite, in particular $M$  has  finite weighted volume.
\end{thm}

\begin{proof}
We have
\begin{equation*} z_i\mathop{=}\limits^{def.} e_i(z)= 2\left<F,
e_i\right>,\end{equation*}
\begin{equation*}
z_{ij}\mathop{=}\limits^{def.} Hess (z)(e_i,e_j) =2 \left(\de_{ij}
-y^\a h_{ij}^\a\right),
\end{equation*}
\begin{equation}
\De z = 2 n - 2 y^\a H^\a =2n+Y^2,
\label{Dez}
\end{equation}
where the self-shrinker equation (\ref{ss}) has been used in third
equality. For our self-shrinker $M^n$ in $\R^{n+m}_m$, we define a
functional $F_t$ on any set $\Om\subset M$ by
$$F_t(\Om)=\f1{(4\pi t)^{n/2}}\int_{\Om}
e^{-\f{z}{4t}}d\mu, \quad \mathrm{for} \quad t>0.$$

Set $B_r=\{p\in\ir{m+n}_m, z(p)<r^2\}$ and $D_r=B_r\bigcap M$.  We
differential $F_t(D_r)$ with respect to $t$,
$$F_t'(D_r)=(4\pi)^{-\f n2} t^{-(\f n2+1)}\int_{D_r} (-\f n2+\f{z}{4t})e^{-\f{z}{4t}}d\mu.$$
Noting (\ref{Dez})
\begin{equation}\aligned
-e^{\f{z}{4t}}\mathrm{div}(e^{-\f{z}{4t}}\n z)&=-\De z+\f1{4t}\n z\cdot\n z\\
&=-2n-Y^2+\f{X^2}t\\
&\ge\f{z}t-2n\quad(\ when\ 0<t\le 1\ ).
\endaligned
\end{equation}
Since $$\n z=2F^T$$ and the unit normal
vector to $\p{D_r}$ is $\f{F^T}{X}$,  then
\begin{equation}\aligned\label{2.2}
F_t'(D_r)\le&\pi^{-\f n2}(4t)^{-(\f n2+1)}\int_{D_r} -\mathrm{div}(e^{-\f{z}{4t}}\n z)d\mu\\
=&\pi^{-\f n2}(4t)^{-(\f n2+1)}\int_{\p D_r} -2Xe^{-\f{z}{4t}}\le 0.
\endaligned
\end{equation}
We integrate $F_t'(D_r)$ over $t$ from $\f{1}{r}$ to $1, \, r\ge 1$,   and get
\begin{equation*}
\int_{D_r}e^{-\f{z}{4}}d\mu\le r^{\f{n}{2}}\int_{D_r}e^{-\f{zr}{4}}d\mu,\label{d005}
\end{equation*}
Since
$$\int_{D_r}e^{-\f{z}{4}}d\mu\ge e^{-\f{r^2}{4}}\int_{D_r}1d\mu$$
and
$$\aligned\int_{D_r}e^{-\f{zr}{4}}d\mu&=&\int_{D_r\backslash D_{\f{r}{2}}}e^{-\f{zr}{4}}d\mu+\int_{D_\f{r}{2}}e^{-\f{zr}{4}}d\mu\\
&\le&e^{-\f{r^3}{16}}\int_{D_r}1d\mu+\int_{D_\f{r}{2}}1d\mu.\label{d006}\endaligned
$$
Set $V(r)=\int_{D_r}1d\mu$. Then,
$$ (e^{-\f{r^2}{4}}-e^{-\f{r^3}{16}}r^\f{n}{2})V(r)\le r^\f{n}{2}V(\f{r}{2}).\label{d007}$$
Let $g(r)=e^{-\f{r^2}{4}}-e^{-\f{r^3}{16}}r^\f{n}{2}$. $g(r)>0$ when
$r$ sufficiently large (say $r\ge 8n$)
Since
$$\aligned
g'(r)&=-\f{r}{2}e^{-\f{r^2}{4}}-\f{n}{2}r^{\f{n}{2}-1}e^{-\f{r^3}{16}}+\f{3r^2}{16}r^{\f{n}{2}}e^{-\f{r^3}{16}}\\
&>(-\f{r}{2}-\f{n}{2}r^{\f{n}{2}-1}+\f{3}{16}r^{\f{n}{2}+2})e^{-\f{r^3}{16}}>0,\endaligned$$
$g(r)$ is increasing in $r$ and $g^{-1}(r)$ is decreasing in $r$. Therefore,
$$
g^{-1}(r)\le\f{1}{e^{-16n^2}-e^{-32n^3}(8n)^{\f{n}{2}}}=C_1\label{d008}$$
We then have
$$V(r)\le C_1r^nV(\f r2) \quad \quad \text{for $r$ sufficiently large
(say, $r\ge8n$)}.$$ By Lemma \ref{DXL}, we have
$$V(r)\le C_4e^{2n(\log r)^2} \quad for \quad r\ge8n,$$ here $C_4$
is a constant depending only on $n,V(8n)$. Hence, for any $\a>0$
\begin{equation*}\aligned
\int_M e^{-\a z}d\mu=&\sum_{j=0}^{\infty}\int_{D_{8n(j+1)}\setminus D_{8nj}}e^{-\a z}d\mu\le\sum_{j=0}^{\infty}e^{-\a (8nj)^2}V(8n(j+1))\\
\le&C_4\sum_{j=0}^{\infty}e^{-\a (8nj)^2}e^{2n(\log (8n)+\log(j+1))^2}\le C_5,
\endaligned
\end{equation*}
where $C_5$ is a constant depending only on $n,V(8n)$. So we obtain
our estimates. Certainly,  $M$ has weighted finite volume.
\end{proof}

\begin{cor}
Any space-like self-shrinker $M$ of dimension $n$ in $\ir{m+n}_m$ with  closed  Euclidean topology has
finite fundamental group.
\end{cor}

From the Gauss equation we have
$$\text{Ric}(e_i,e_i)=\left<H,B_{ii}\right>-\sum_j\left<B_{ij},B_{ij}\right>,$$
and
$$\text{Hess}(f)(e_i,e_i)=\f{1}{4}\text{Hess}(z)(e_i,e_i)=\f{1}{2}\de_{ij}+\f{1}{2}\left<F,B_{ij}\right>=\f{1}{2}\de_{ij}-\left<H,B_{ij}\right>,$$
It follows that
$$\text{Ric}_f(e_i, e_i)=\text{Ric}(e_i,e_i)+\text{Hess}(f)(e_i,e_i)\ge\f{1}{2}.$$
Set $B_R(p)\subset M$, a geodesic ball of radius $R$ and centered at $p\in M$. From Theorem 3.1 in \cite{W-W} we know that for any $r$ there are constant $A,\;B$
and $C$ such that
\begin{equation}
\int_{B_R(p)}\rho\le A+B\int_r^R e^{-\f{1}{2}t^2+Ct}dt\label{WW}
\end{equation}

\bigskip

\Section{Rigidity results}{Rigidity  results}
\medskip

Now, we are in a position to prove  rigidity results mentioned in
the introduction.

\begin{thm}
Let $M$ be a space-like self-shrinker of dimension $n$ in
$R^{n+m}_m,$ which is closed with respect to the Euclidean topology.
If there is a constant $\a<\f{1}{8}$, such that $|H|^2\le e^{\a z}$, then $M$
is an affine $n-$plane.
\end{thm}
\begin{proof}
Let $\e$ be a smooth function with compact support in $M,$ then by
(\ref{dLH}) we obtain
\begin{equation}\aligned\label{5.1}
\int_M(\f{1}{2}|H|^2+|P|^2+|\n
H|^2)\e^2\rho&=\f{1}{2}\int_M(\mathcal{L}|H|^2)\e^2\rho=\f{1}{2}\int_M\text{div}(\rho\nabla|H|^2)\e^2\\
&=-\int_M\e\rho\n|H|^2\cdot\n\e\\&=2\int_M\e\rho\langle\n_i H,H\rangle\cdot\n_i\e\\
&\le\int_M|\n
H|^2\e^2\rho+\int_M|H|^2|\n\e|^2\rho.\endaligned\end{equation} We
then have
\begin{equation}\label{5.2}
\int_M \left(\f{1}{2}|H|^2+|P|^2\right)\e^2\rho \le\int_M |H|^2|\n\e|^2\rho.
\end{equation}

Let $\e=\phi(\f{|F|}{r})$ for any $r>0,$ where $\phi$ is a
nonnegative function on $[0,+\infty)$ satisfying
\begin{equation*}\aligned
\phi(x)=\left\{\begin{array}{lllll}1&\quad& &if& x\in [0,1)\\0&\quad& &if& x\in[2,+\infty),\end{array}\right.\endaligned
\end{equation*}
and $|\phi'|\le C$ for some absolute constant. Since $\n z=2F^T$,
$$\n\e=\frac 1r\phi'\n\sqrt{z}=\frac 1r\phi'\frac{F^T}{\sqrt{z}}.$$
By (\ref{ss}) we have
$$|\n\e|^2\le\frac{C^2}{r^2}\frac{|F^T|^2}{z}=\f{1}{r^2z}C^2(z+4|H|^2).$$
It follows that (\ref{5.2}) becomes
\begin{equation}\label{5.3}
\int_{D_r}\left(\f{1}{2}|H|^2+|P|^2\right)\rho\le\f{C^2}{r^2}\int_{D_{2r}\setminus
D_r}|H|^2(1+\frac{4|H|^2}{z})\rho.
\end{equation}

By  Theorem \ref{vg}  then under the condition on  $|H|$, we obtain
that the right hand side of (\ref{5.3}) approaches to zero as
$r\rightarrow +\infty.$ This implies that $H\equiv0.$

According to Theorem 3.3 in \cite{J-X} we see that $M$ is complete
with respect to the induced metric from $\ir{m+n}_m$. In a geodesic
ball $B_a(x)$ of radius $a$ and centered at $x\in M$ we can make
gradient estimates of $|B|^2$ in terms of the mean curvature. From
(2.9) in \cite{J-X} we have
$$|B|^2\le k\f{2m(n-4)a^2}{(a^2-r^2)^2}.$$
Since $M$ is complete we can fix $x$ and let $a$ go to infinity.
Hence, $|B|^2=0$ at any $x\in M$ and $M$ is an $n-$plane.
\end{proof}

\begin{thm}
Let $M$ be a complete space-like self-shrinker of dimension $n$ in
$R^{n+m}_m$. If there is a constant $\a<\f{1}{2}$, such that $\ln w\le e^{\a d^2(p, x)}$ for certain
$p\in M$, where $d(p, \cdot)$ is the distance function from $p$, then $M$ is affine $n-$plane.
\end{thm}

\begin{proof}
(\ref{dLlogw}) tells us
$$\mc{L}(\ln w)\ge\f{|B|^2}{w^2}\ge 0.$$
As an application to (\ref{WW}) (Theorem 3.1 in \cite{W-W}) the Corollary 4.2 in \cite{W-W} tells us that $\ln w$ is constant. This forces $|B|^2\equiv 0$.
\end{proof}

\bigskip

\bibliographystyle{amsplain}

\end{document}